\documentclass[12pt]{article}
\usepackage{amsthm,amsmath}
\begin{document}
\newtheorem{thm}{Theorem}[section]
\newtheorem{lem}[thm]{Lemma}
\newtheorem{cor}[thm]{Corollary}
\theoremstyle{definition}
\newtheorem{defn}[thm]{Definition}
\newtheorem{conj}{Conjecture}
\newtheorem{quest}[conj]{Question}
\newtheorem*{thm*}{Theorem}
\theoremstyle{remark}
\newtheorem*{remark}{Remarks}
\newcommand{\sgn}{\mbox{sgn}}
\title{Integrality and Specialized Symmetric Functions}
\author{Allan Berele\\Department of Mathematical
Sciences\\DePaul University\\Chicago, IL 60614}
\maketitle
\section{Introduction}
In \cite{b24}, together with S.~Catoiu, we introduced the specialized symmetric functions.  These
can be regarded from two points of view, either symmetric function
theory or trace idenitities.  Given $(d)=(d_1,\ldots,d_k)$ we
defined the specialized power function as $p^{(d)}_n(x_1,\ldots,
x_k)$ to be $d_1x_1^n+\cdots+d_kx_k^n$.  In the special
case in which the $d_i$ are all positive integers, $p^{(d)}_n$
can be gotten from the ordinary power symmetric functions
by specializing certain groups of variables to be equal, hence the name.
We will sometimes surpress the superscript $(d)$ since
we have no need to investigate the classical case corresponding to $(d)=(1,\ldots,1)$.
There are corresponding analogues of the other classical
symmetric functions: The elementary symmetric functions,
the complete symmetric functions and the Schur functions,
and these satisfy most of the relations satisfied by their
more classical counterparts, although our focus in the
current work will only be on the specialized power 
symmetric functions.
As is standard in symmetric function theory, given a partition 
$\lambda=(\lambda_1,\lambda_2,\ldots)$
we denote the product $p_{\lambda_1}p_{\lambda_2}\cdots$
as $p_\lambda$.  Our main interest will be in the algebraic relations
satisfied by the $p_\lambda$.  

From the trace identity point of view, let $A=F^k$, be $k$
copies of the field, sometimes considered as diagonal $k\times k$
matrices.  Since $A$ is commutative, the trace functions on $A$
are simply the linear functionals, and so every non-degenerate
trace function on $A$ is of the form
$$tr(a_1,\ldots,a_k)=d_1a_1+\cdots d_ka_k,$$
for some non-zero $d_1,\ldots,d_k$ in the field.
We use the notation $D{(d)}$ for the corresponding algebra of
diagonal matrices with trace. The 
study of the trace identities of $D{(d)}$ was initiated by Ioppolo, Koshlukov and
La Mattina in \cite{I21}, in which they found bases for 
the ideal of identities when $k\le 2$. If $X=
(x_1,\ldots,x_k)$, then the trace of $X^n$ is  $d_1x_1^n+\cdots
d_kx_k^n$, which is $p^{(d)}_n(x_1,
\ldots,x_k)$.  Let $\Lambda{(d)}\subset F[x_1,\ldots,x_k]$ be
the subring generated by the $p^{(d)}_n$.  Again,
we will surpress the $(d)$ when we don't feel
that it is necessary.  We will also write $\Lambda_n$
or $\Lambda_n{(d)}$ for the algebra generated
by $\{p_1,\ldots,p_n\}$.  We showed in~\cite{b24} that
$\Lambda{(d)}$ has transcendence degree $k$, so any set
of more than $k$ of the $p_n$ will satisfy algebraic relations.  In the classical
case in which each $d_i=1$, the Cayley-Hamilton theorem
says that $X$ is monic algebraic over $\Lambda$ of degree
$n$.  From the point of view of trace identities, the 
Cayley-Hamilton theorem is a mixed trace identity (i.e., and identity involving both
trace and non-trace terms) in one
variable.  We would like to know what the trace identities
are in this more general case.    In the current work we will be
focusing mainly on trace identities in one variable.  So,
 pure trace identities in one variable are  are relations between
the $p_\lambda$; and an analogue of Cayley-Hamilton
would be a mixed trace identity in one variable.

We define absolute identities to be ones which depend on $k$, but not otherwise on $(d)$.
In \cite{b24} we proved that all of the pure, absolute one variable  identities are consequences of the determinental identities,
$|p_{a_i+b_j}|=0$, for any $a_1,\ldots,a_{k+1}$ and $b_1,\ldots,
b_{k+1}$.  We also proved that $X$ was algebraic over $\Lambda
$, but we didn't prove that it was monic.  \cite{b24} posed six
questions or conjectures.  In this paper we will address the 
following two of them:
\begin{description}
\item[Conjecture 1:] $X$ is integral over $\Lambda$.

\item [Conjecture 3:]The algebra of $k\times k$ diagonal matrices
satisfies an identity of the form $tr(y_1)\cdots tr(y_m)=$
a linear combination of trace terms each a product of fewer than $m$ 
traces, for a certain specified $m$.  
\end{description}
We will prove that Conjecture 1 is true generically and that
Conjecture~3 is true.  By generically we mean that
diagonal matrices satisfy an identity of the form 
$$\alpha_n X^n=\sum \alpha_\lambda p_\lambda 
X^{n-|\lambda|},$$
in which the $\alpha$ are in $F[d_1,\ldots,d_k]$.  However,
for certain values of the $d_i\in F$, the coefficient $\alpha_n$
of $X^n$ may become zero.  Specifically, if  $\sum\{d_i|i\in I\}$
is zero for some $I\subseteq\{1,\ldots,k\}$, then $X$ will not be
integral over $\Lambda$.  We now conjecture that $X$ will be integral
in all othe cases.

Conjecture 3 has the consequence that any product of
$m$ or more traces can be written as a linear combination
of terms with at most $m-1$.  Conjecture 1 has a 
somewhat similar application:   If $X$ is integral of degree $m$ over $\Lambda_{m}$,
then $\Lambda_{m}=\Lambda$.  Because if
$$X^m=\sum a_\lambda p_\lambda X^{m-|\lambda|}$$
where $a_\lambda\in F[d_1,\ldots,d_k]$, then if we multiply by 
$X^i$, $i\ge1$, and take the trace of both sides of the equation, we get
$$p_{m+i}=\sum a_\lambda p_\lambda p_{m-|\lambda|+i}$$
so $p_{m+i}\in \Lambda_{m+i-1}$, and by induction $p_{m+i}\in \Lambda_{m}$.

There is also a Conjecture~2 in \cite{b24} that we have not
succeeded in proving, but it seems germain to mention it.  Let
$(d)$ have $t$ parts with multiplicities $m_1,\ldots,m_t$ and let
$m=\prod(m_i+1)$.  Conjecture~2 says that the degree of
integrality in Conjecture~1 should be $m-1$ and the number
of traces in Conjecture~2 should be $m$.

We gratefully acknowledge the useful conversations we had with
S.~Catoiu  in the course of this investigation.
\section{Non-Integrality}
The ideal of (pure or mixed) trace idenities is homgeneous.  Henceforth,
whenever we prove the existence or non-existence of identities we will
assume homogeneity.

We first show that if there is a set $I\subseteq \{1,\ldots,k\}$
such that $\sum_{i\in I} d_i=0$, then $X$ cannot be integral
over $\Lambda$.
\begin{thm} If $\sum_{i\in I} d_i=0$, then $X$ cannot be integral
over $\Lambda$.\label{th:1}
\end{thm}
\begin{proof} Assume to the contrary that
$$X^N+\sum a_\lambda p_\lambda X^i=0,$$
where each $i<N$, and so each $\lambda$ in the sum has
degree $N-i>0$.  If we specialize $x_i=\delta(i\in I)x$, where
$\delta(i\in I)$ is 1 or 0, depending on whether $i\in I$ is true,
then each $p_j=\sum_m d_mx_m^j$ becomes 0, but $X^N$ does
not.
\end{proof}

In the case in which each $d_i$ is either plus or 
minus 1, the specialized symmetric functions are
related to the representations of the general linear
Lie algebra.  If $(d)$ has $k$ ones and $\ell$ minus
ones then the diagonal matrices $D{(d)}$ are a trace
subalgebra of the Kemer algebra, $M_{k,\ell}$.
The trace identities of $M_{k,\ell}$ were found in \cite{b15}, although
Kantor and Trishin had already found the one
variable identities in~\cite{k99}.  Translating
their results to the language of specialized
symmetric functions, given
a partition $\sigma$ of cycle type $(a_1,\ldots,a_k)$
let $T(\sigma)$ be $p_{a_1}\cdots p_{a_k}$.  
\begin{thm} Let
$\lambda$ be the partition $(k+1,\ldots,k+1)$ of 
$n=(k+1)(\ell+1)$.
Then the group algebra $FS_{n}$ has a two-sided ideal
$I_\lambda$, and for every $e\in I_\lambda$, $T(e)$
is an identity in $\Lambda$, and all identities for
$\Lambda$ are consequences of these.
\end{thm}
\section{Initialization Step}

\begin{defn} We sill say that a homogeneous polynomial $f\in F[d_1,
\ldots,d_k;\allowbreak p_1,p_2,\ldots]$ is \emph{sufficiently monic\/}
if homogeneous in total degree and if
 it is of the form $\sum_{\lambda\vdash N}
 \alpha_\lambda p_\lambda$, where each $\alpha_\lambda$
is in $F[d_1,\ldots,d_k]$, $\alpha_{(N)}\ne 0$ and $\alpha_{(1^N)}=1$. I.e., $f$ is of the form
$$\alpha(d_1,\ldots,d_k)p_N+\cdots+p_1^N.$$

 We will say that $\Lambda{(d)}$
satisfies such an $f$ if $f$ vanishes under the substitutions
$p_i\mapsto d_1x_1+\ldots+d_kx_k$.
\end{defn}
\begin{defn}
Likewise, we will say that a homogeneous polynomial $f$ in $F[d_1,
\ldots,d_k;\allowbreak p_1,p_2,\ldots,t]$ is \emph{
sufficiently monic\/} if it is honogeneous in total degree and if it is of the form
$\sum \alpha_\lambda p_\lambda t^{N-|\lambda|}$,
where each $\alpha_\lambda$
is in $F[d_1,\ldots,d_k]$, $\alpha_\emptyset\ne 0$ and $\alpha_{(1^N)}=1$. I.e., $f$ is of the form
$$\alpha_{\emptyset}(d_1,\ldots,d_k)t^N+\cdots+
p_1^N.$$

 We will say that $x$
satisfies such an $f$ if $\Lambda(d)$ satisfies $f\vert_{t=x}$.
We will also say that $x$ is sufficiently monic over $\Lambda(d)$
in this case.
\end{defn}
The main step in the remaining proofs is to prove
that each of $\Lambda^{(d)}$ and $X$ satisfies  sufficiently
monic polynomials, which we will do by induction
on~$k$, the length of~$(d)$.  It turns out that 
either of these two conditions implies the other,
but we will only need that a  monic polynomial satisfied by $X$
over $\Lambda$ implies a relation amongst the $p_\lambda$.
The obvious method would be to take traces, but there is 
a technical problem.  Consider the following identity satisfied
by $D{(a,b)}$:
\begin{multline*}
ab(a+b)X^3-3abp_1X^2-(a^2-ab+b^2)X+(a+b)p_1^2X\\
-(abp_3+(a+b)p_2p_1-p_1^3)I=0
\end{multline*}
If we take the trace of both sides, using $tr(X^i)=p_i$ and
$tr(I)=a+b$, we simply get $0=0$.  The problem disappears if
we multiply by $X$ before taking trace, and this works in
general.
\begin{lem} Let $f=\sum a_\lambda p_\lambda X^{N-|\lambda|}$
be a sufficiently monic element of $\Lambda[X]$.  Let $g=tr(X\cdot f)$,
namely $\sum a_\lambda p_\lambda p_{N+1-|\lambda|}$.
Then the coefficient of $p_{N+1}$ in $g$ is $a_\emptyset\ne0$ and the 
coefficient of $p_{1^{N+1}}$ in $g$ is $a_{1^N}=1$, and so $tr(X\cdot f)$ is sufficiently 
monic.  \label{lem:3.3}
\end{lem}

We now turn to an induction proof of the existence
of sufficiently monic polynomial identities.  For the
initialization step, we jump over the trivial case of
$k=1$ to do $k=2$.  This actually was done by
Ioppolo, Koshlukov and La Mattina in \cite{I21},
but we choose a different proof (leading to a higher
degree!) because it illustrates the process we use in
the induction step.  We start with this easy observation.
\begin{lem} The polynomial $\sum a_\lambda p_\lambda X^i$
is an identity for $D{(d)}$ if and only if $\sum a_\lambda p_\lambda
x_j^i=0$ for all $j=1,\ldots,k$.\label{lem:3.4}
\end{lem}
\begin{proof}
$\sum a_\lambda p_\lambda X^i$ is a diagonal matrix and
the $j^{th}$ entry is $\sum a_\lambda p_\lambda
x_j^i$.
\end{proof}
In the case of $k=2$, we write
$a$ and $b$ in place of $d_1$ and $d_2$, and $x$ and $y$
in place of $x_1$ and $x_2$.
So
$$p_1=ax+by\mbox{ and }p_2=ax^2+by^2.$$
Hence, $ax=p_1-by$ and $ax^2=p_2-by^2$, and since
$(ax)^2=a(ax^2)$ we have the identity
$(p_1-by)^2=a(p_2-by^2),$ which simplifies to
$$b(a+b)y^2-2bp_1y+(p_1^2-ap_2)=0.$$
By a similar argument,
$$a(a+b)x^2-2ap_1x+(p_1^2-bp_2)=0.$$
Setting $x=y=t$ in each and multiplying these two equations gives an equation over $\Lambda$ which vanishes if $t=x$ or~$y$, hence,
by Lemma~\ref{lem:3.4} it is a trace identity for $D{(a,b)}$.  We observe that the
coefficient of $X^4$ is $ab(a+b)^2$ will be non-zero unless $a$,
$b$ or $a+b$ is zero.  We also observe the occurence of
$p_1^4$.  In sum:
\begin{lem} When $(d)=(a,b),$ $X$ satisfies a sufficiently
monic identity of the form
$$ab(a+b)^2X^4+\cdots+p_1^4$$\label{lem:3}
\end{lem}
Multiplying by $X$ and taking traces  as in Lemma~\ref{lem:3.3}, gives
\begin{cor} When $(d)=(a,b)$, $\Lambda{(d)}$
satisfies a sufficiently monic identity of the form
$$ab(a+b)^2p_5+\cdots+p_1^5.$$
\end{cor}
We chose to multiply by $X$ before taking traces
because, had we not, the last term would have
been multiplied by $a+b$, the trace of the identity
matrix, in which case the $p_1^4$ would have
vanished in the $a+b=0$ case.  As it is, it is valid
for all $a$ and $b$.
\section{The Induction Step}
In this section we assume that for every $(d)$
with $k$ parts, the diagonal matrix $X$ is sufficiently monic over
$\Lambda{(d)}$ and that $\Lambda{(d)}$
satisfies a sufficiently monic polynomial.  We
let $(d^+)=(d_1,\ldots,d_{k+1})$ have $k+1$
parts and for each $1\le i\le k+1$ we let $(\hat{d}_i)$
be $(d^+)$ with $d_i$ removed.  The strategy is
to construct sufficiently monic polynomials 
$g_i(t)\in \Lambda{(d^+)}[t]$
such that $g_i(x_i)=0$.  Then their product
will be an identity for $X$ by Lemma~\ref{lem:3.4}.
To save notation, we will write $p_\lambda$ for
$p_\lambda^{(d)}$ and $q_\lambda$ for $p_\lambda^{(d^+)}$.  
Given $f\in F[d_1,\ldots,d_k,p_1,p_2,\ldots]$, for each $1\le i\le k+1$
we define $g_i\in F[d_1,\ldots,d_{k+1},q_1,q_2,\ldots,t]$ via
$$g_i(t)=f(\hat{d}_i,q_1-d_it,q_2-d_it^2,\ldots].$$

\begin{lem} If $f$ is an identity for $\Lambda{(d)}$, then
each $g_i(x_i)$ is an identity for $\Lambda{(d^+)}$.
\end{lem}
\begin{proof} If $t=x_i$, then each $q_j-d_it^i$ becomes 
$\sum_{\alpha\ne i}d_\alpha x_\alpha$ which equals $p_j^{(\hat{d}_i)}$,
so $g_i(x_i)=f(\hat{d}_i,p_i^{(\hat{d}_i)})=0$.
\end{proof}
At this point it is useful to see specifically how $g_i$ is related to $f$.
\begin{lem} If $f=p_\lambda=p_{\lambda_1}\cdots p_{\lambda_a}$,
then $g_i$ equals
$$(q_{\lambda_1}-d_it^i)\cdots (q_{\lambda_a}-d_it^{\lambda_a}).$$
In particular, the highest degree term in $t$ is $(-d_i)^a t^{|\lambda|}$
and the lowest degree term is $q_\lambda$.\label{lem:4.2}
\end{lem}
The number of non-zero parts to a partition
$\lambda$ is called the length of $\lambda$ and
is denoted $\ell(\lambda)$.
\begin{cor} If $f=\sum_{\lambda\vdash N} \alpha_\lambda p_\lambda$, then the term in $g_i(t)$ of highest degree in $t$ is
$\sum_{\lambda\vdash N} \alpha_\lambda(-d)^{\ell(\lambda)}$.
\end{cor}
In the case of sufficiently monic polynomials we have
\begin{lem} If $f=\sum_{\lambda\vdash N} \alpha_\lambda p_\lambda$ is sufficiently monic, then so is each $g_i$.
\end{lem}
\begin{proof} By Lemma~\ref{lem:4.2}, the terms of  $g_i$ 
not involving $t$ will be $g_i(\hat{d}_i,q_1,q_2,\ldots)$, and so
if $f$ has a summand of $p_1^N$, then $g_i$ will have a summand
of $q_1^N$.  And, if $f=\alpha_N(d) p_N+\cdots$, then
the highest degree term of $g_i$ in $t$ will be $t^N$ times
$(-d_i)\alpha_N(\hat{d}_i)+$ higher degree terms in $d_i$.
\end{proof}
\begin{thm} Let $f$ be a sufficiently monic
identity for $X$ over $\Lambda{(d)}$.  Then
the product $g_1(X)\cdots g_{k+1}(X)$ will be a
sufficiently monic identity for $X$ over
$\Lambda{(d^+)}$.\label{th:4.5}
\end{thm}
This proves Conjecture~1.  Conjecture~2 follows from
multilinearization:
\begin{proof}[Proof of Conjecture 2] Theorem \ref{th:4.5} implies
that the generic diagonal
polynomial $X$ will satisfy a identity of the form
$$tr(X)^N+\mbox{terms involving fewer traces}=0.$$
We can now multilinearize:  Let $X=c_1y_1+\cdots+
c_Ny_n$ and take the coefficient of $c_1\cdots c_N$,
getting
$$tr(y_1)\cdots tr(y_N)+\mbox{terms involving fewer traces}=0.$$
\end{proof}


\begin{thebibliography}{9}
\bibitem{Bel} G. Bellavitis, Teorica dei determinanti, {\it Memorie Dell' I. R. Instituto Veneto di Scienze lettera ed arti \bf7} (1857), 75--136.  Available on-line at  https://babel.hathitrust.org/cgi/pt?id=mdp.39015009239644\&seq=75\&q1=bellavitis
\bibitem{b15} A. Berele, $J$-trace identities and invariant
theory, {\it Advances in Applied Math. \bf63} (2015),
1--18.
\bibitem{b24} A. Berele and S. Catoiu, Specialized symmetric functions,
{\it Israel J. Math. \bf259} (2024), 503--526.
\bibitem{I21} A. Ioppolo, P. Koshlukov and D. La Mattina, Trace
identities and almost polynomial growth, {\it J. Pure and Appl.
Algebra \bf 255} (2021).
\bibitem{k99} I. Kantor and I. Trishin, On Cayley-Hamilton
equation in the supercase, {\it Comm.  Algebra, \bf27} (1999),
233-259.
\bibitem{L} A. Lascoux, Symmetric Functions and Combinatorial
Operators on Polynomials, CBMS Regional Conference Series
in Mathematics, no. 99, Amer. Math. Soc., Providence, R.I. (2003)
\end{thebibliography}
\end{document}